\date{}
\title{Erd\H{o}s--Szekeres-type statements:\\
Ramsey function and decidability in dimension~1\thanks{Both authors
were supported by the ERC Advanced Grant No.~267165.
J.\,M. was also supported by the project P202/12/G061 of GA~\v{C}R.}}

\author{
{\sc Boris Bukh}\\
  {\footnotesize Centre for Mathematical Sciences}\\[-1.5mm]
  {\footnotesize Wilberforce Rd, University of Cambridge}\\[-1.5mm]
  {\footnotesize Cambridge CB3 0WB, Great Britain and}\\
  {\footnotesize Churchill College, Storey's Way}\\[-1.5mm]
  {\footnotesize Cambridge CB3 0DS, Great Britain}
\and
{\sc Ji\v{r}\'{\i} Matou\v{s}ek}\\
   {\footnotesize Department of Applied Mathematics}\\[-1.5mm]
   {\footnotesize  Charles University, Malostransk\'{e} n\'{a}m. 25}\\[-1.5mm]
{\footnotesize  118~00~~Praha~1,
   Czech Republic, and}\\
{\footnotesize    Institute of  Theoretical Computer Science}\\[-1.5mm]
{\footnotesize    ETH Zurich,
      8092 Zurich, Switzerland}
}

\newcommand{\cmt}[1]{\ifhmode\newline\fi{\sf *** \ \ #1 \\}}

\documentclass[11pt]{article}

\usepackage{a4wide}

\ifdefined\marginlesspreview
\usepackage[papersize={130mm,190mm},text={127mm,186mm}]{geometry}
\fi

\usepackage{amsmath,amsthm,amssymb,tabularx,enumitem}
\usepackage{graphicx}
\usepackage{url}

\usepackage[colorlinks=true,citecolor=blue,urlcolor=blue,linkcolor=blue,bookmarksopen=true]{hyperref}

\newtheorem{theorem}{Theorem}[section]
\newtheorem{definition}[theorem]{Definition}
\newtheorem{prop}[theorem]{Proposition}
\newtheorem{observation}[theorem]{Observation}
\newtheorem{lemma}[theorem]{Lemma}
\newtheorem{corollary}[theorem]{Corollary}

\newcommand*{\R}{\mathbb{R}}                          
\newcommand*{\C}{{\mathbb{C}}}
\newcommand*{\N}{{\mathbb{N}}}

\newcommand*{\FF}{{\mathcal{F}}}
\newcommand*{\GG}{{\mathcal{G}}}
\newcommand*{\DD}{{\mathcal{D}}}
\newcommand*{\QQ}{{\mathcal{Q}}}

\newcommand\makevec[1]{{\bf #1}}
\newcommand*{\xx}{{\makevec x}}
\newcommand*{\yy}{{\makevec y}}
\renewcommand*{\aa}{{\makevec a}}
\newcommand*{\bb}{{\makevec b}}

\newcommand\bPhi{{\makevec\Phi}}

\newcommand*{\seq}[1]{{\underline{#1}}}

\DeclareMathOperator{\len}{len}                      
\DeclareMathOperator{\rev}{rev}                       

\DeclareMathOperator{\sgn}{sgn}

\DeclareMathOperator{\ES}{ES}
\DeclareMathOperator{\conv}{conv}

\newcommand{\alterdef}[1]{\!\left\{\!\!\begin{array}{ll}
                                   #1 \end{array}  \right. }

\newcommand{\heading}[1]{\vspace{1ex}\par\noindent{\bf #1}}

\renewcommand\:{\colon}

\long\def\onefigure#1#2{
\begin{figure*}[tbp]
\begin{center}
#1
\end{center}
\caption{#2}
\end{figure*}
}

\def\immediateFigure#1{%
\smallskip\begin{center}#1\end{center}\smallskip }

\newcommand{\labfig}[2]  
{\onefigure{\mbox{\includegraphics{#1}}}{\label{f:#1} #2} }

\newcommand{\labfigw}[3]  
{\onefigure{\mbox{\includegraphics[width=#2]{#1}}}{\label{f:#1} #3}}

\newcommand{\immfig}[1]  
{\immediateFigure{\mbox{\includegraphics{#1}}}}

\newcommand{\immfigw}[2] 
{\immediateFigure{\mbox{\includegraphics[width=#2]{#1}}}}

\begin{document}

\maketitle

\begin{abstract} 
A classical and widely used lemma of Erd\H{o}s
and Szekeres asserts that for every $n$ there exists $N$ such
that every $N$-term sequence $\seq a$ of real numbers
contains an $n$-term increasing subsequence or an $n$-term
nonincreasing subsequence; quantitatively, the smallest
$N$ with this property equals $(n-1)^2+1$. In the setting of the present
paper, we express this lemma by saying that the set
of predicates $\bPhi=\{x_1<x_2,x_1\ge x_2\}$ is 
Erd\H{o}s--Szekeres with Ramsey function $\ES_\bPhi(n)=(n-1)^2+1$.

In general, we consider an arbitrary finite set $\bPhi=
\{\Phi_1,\ldots,\Phi_m\}$ of \emph{semialgebraic predicates},
meaning that  each $\Phi_j=\Phi_j(x_1,\ldots,x_k)$ 
is a Boolean combination of polynomial equations and inequalities
in some number $k$ of real variables. We define
$\bPhi$ to be \emph{Erd\H{o}s--Szekeres} if for every $n$
there exists $N$ such that  each $N$-term sequence $\seq a$
of real numbers has an $n$-term subsequence $\seq b$
such that at least one of the $\Phi_j$ holds everywhere on $\seq b$,
which means that $\Phi_j(b_{i_1},\ldots,b_{i_k})$ holds
for every choice of indices $i_1,i_2,\ldots,i_k$, $1\le i_1<i_2<\cdots
<i_k\le n$. We write $\ES_\bPhi(n)$ for the smallest $N$
with the above property.

We prove two main results. First, the Ramsey
functions in this setting are at most doubly exponential
(and sometimes they are indeed doubly exponential):
for every $\bPhi$ that is Erd\H{o}s--Szekeres, there
is a constant $C$ such that $\ES_\bPhi(n)\le 2^{2^{Cn}}$.
Second, there is an algorithm that, given $\bPhi$, decides
whether it is Erd\H{o}s--Szekeres; thus, one-di\-mension\-al
Erd\H{o}s--Szekeres-style theorems can in principle
be proved automatically. 

We regard these results as a starting point in investigating
analogous questions for $d$-di\-mension\-al predicates, 
where instead of sequences of real numbers,
we consider sequences of points in $\R^d$ 
(and semialgebraic predicates in their coordinates).
This setting includes many results and problems
in geometric Ramsey theory, and it appears considerably more involved.
Here we prove a decidability result for 
 \emph{algebraic} predicates in $\R^d$ (i.e., conjunctions
of polynomial equations), as well as for a multipartite
version of the problem with arbitrary semialgebraic
predicates in~$\R^d$.
\end{abstract}

\section{Introduction}

\heading{Motivation and background. } 
Ramsey-type theorems claim that, generally speaking, any sufficiently 
large structure of a given kind
contains a ``very regular'' substructure of a prescribed size.
The following two gems from a 1935 paper of Erd\H{o}s and Szekeres
\cite{es-cpg-35} belong among the earliest, best known, and most
useful instances.

\begin{theorem}[On monotone subsequences]\label{t:es-monot}
For every $n\ge 2$, every sequence $(a_1,a_2,\ldots,a_N)$
of real numbers, with $N\ge (n-1)^2+1$, contains a monotone
subsequence of length $n$; more precisely, there are
indices  $i_1<i_2<\cdots<i_n$ such that either
$a_{i_1}<\cdots<a_{i_n}$ or $a_{i_1}\ge \cdots\ge a_{i_n}$.
\end{theorem}

See, for example, Steele \cite{steele-surv} for a collection
of six nice proofs and many applications.

\begin{theorem}[On subsets in convex position]\label{t:es-conv}
For every $n\ge 3$,
among every $N={2n-4\choose n-2}+1\le 4^n$
points in the plane, no three collinear, one can always
select $n$ points in convex position (i.e., forming the
vertex set of a convex $n$-gon).
\end{theorem}

See, e.g., \cite{MorrisSoltan,Mat-dg} for proofs and surveys
of developments around this result.

Many geometric Ramsey-type questions have been investigated in
the literature; here we mention just three examples that directly
motivated our research. (We refer to \cite[Chap.~9]{Mat-dg}
for references to many other geometric Ramsey-type results somewhat related
to our line or research, and to Alon et al.~\cite{AlonPPRS05}
Fox et al.~\cite{fox-al-geomexp} for a sample of 
recent work.)

The \emph{colored Tverberg theorem}, conjectured by B\'ar\'any,
F\"uredi, and Lov\'asz \cite{bfl-nhp-90} and proved by 
 Vre{\'c}ica and {\v Z}ivaljevi{\'c} \cite{vz-ctpci-92}
asserts that \emph{for every $d$ and $r$ there exists $t$ such
that if $A_1,\ldots,A_{d+1}$ are $t$-point sets in $\R^d$
(we imagine that each $A_i$ has its own color; e.g., the points
of $A_1$ are red, those of $A_2$ blue, etc.),
there are pairwise disjoint $(d+1)$-point 
sets $B_1,B_2,\ldots,B_r\subseteq \bigcup_{i=1}^{d+1} A_i$
such that $|B_j\cap A_i|=1$ for every $i,j$ (i.e., each $B_j$
uses every color exactly once),
and $\bigcap_{j=1}^r\conv(B_j)\ne\emptyset$.}
For us, this result is remarkable because all known proofs use
topological methods---there is no known proof by ``ordinary'' geometric
and/or combinatorial arguments (although some special cases do have
elementary proofs).

The following result was needed 
as a lemma in the paper \cite{bukh-al-sensitive-apx}
by Loh, Nivasch, and the first author: \emph{every sequence
$(\aa_1,\aa_2,\ldots,\aa_N)$ of points in $\R^2$ with increasing 
$x$-coordinates has a subsequence $(\bb_1,\ldots,\bb_n)$
in which every $7$-tuple
$\bb_{i_1},\ldots,\bb_{i_7}$, $i_1<\cdots<i_7$, is
as in the next picture, i.e., the triangle $\bb_{i_2}\bb_{i_4}\bb_{i_6}$
contains the intersection of the segments $\bb_{i_1}\bb_{i_5}$
and $\bb_{i_3}\bb_{i_7}$---provided that $N$ is sufficiently large
in terms of~$n$.}
\immfig{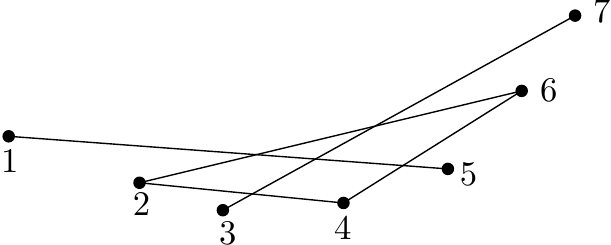}
While the proof in \cite{bukh-al-sensitive-apx} is simple,
the situation with an appropriate $d$-di\-mension\-al generalization
is discouraging: there are increasingly complicated proofs
up to dimension~$4$, while the $5$-di\-mension\-al case already seems
out of reach with the present methods.

Yet another example we want to mention is a Ramsey-type result
of Eli\'a\v{s} and the second author \cite{eli-mat}: \emph{every
$N$-point sequence $(\aa_1,\ldots,\aa_N)$ of points in the
plane with increasing $x$-coordinates has an $n$-term subsequence
that is \emph{$k$th order monotone}, meaning that either every $(k+1)$-tuple
of points lies on the graph of a smooth function with nonnegative $k$th
derivative, or every $(k+1)$-tuple lies on the graph of a smooth function 
with nonpositive $k$th derivative.} (This is a common generalization
of Theorems~\ref{t:es-monot} and~\ref{t:es-conv}.) Here $N$ depends
on $n$ and $k$; the existence of \emph{some} $N$ follows immediately
from Ramsey's theorem for $(k+1)$-tuples, but an interesting question
here is the behavior of the Ramsey function---how big is the smallest $N=
N_k(n)$ that works. We have $N_1(n)$ of order $n^2$ and $N_2(n)$
of order roughly $4^n$
according to Erd\H{o}s and Szekeres, and in \cite{eli-mat}
it was proved that $N_3(n)$ is doubly exponential, i.e.,
$2^{2^{c_1 n}}\le N_3(n)\le 2^{2^{c_2n}}$. The order of magnitude
of $N_4(n)$ is unknown, and so are the Ramsey functions for numerous
other geometric Ramsey-type results.

\heading{Erd\H{o}s--Szekeres predicates. } The above examples and
some others motivate (at least) three general questions formulated
below. Before stating them, we introduce some notions and notation.

Let $k$ be an integer, which we think of as small and fixed,
and let $\Phi=\Phi(\xx_1,\ldots,\xx_k)$ be a $d$-di\-mension\-al
$k$-ary predicate, by which we mean
a mapping $(\R^d)^k\to \{{\rm False},{\rm True}\}$.
We say that $\Phi$ \emph{holds everywhere} on a sequence
$\seq\aa=(\aa_1,\ldots,\aa_n)$ of points in $\R^d$ if
 $\Phi(\aa_{i_1},\aa_{i_2},\ldots,\aa_{i_k})$ holds for every
increasing $k$-tuple $i_1,\ldots,i_k$ of indices,
$1\le i_1<i_2\cdots<i_k\le n$.

\begin{definition}\label{d:esp}
Let $\bPhi$ be a set of $d$-di\-mension\-al predicates.
We say that $\bPhi$ is \emph{Erd\H{o}s--Szekeres}
if for every $n$ there exists $N$ such that every sequence
$\seq \aa$ of length $N$  has a subsequence
$\seq \bb$ of length $n$ such that there is a predicate $\Phi\in\bPhi$
that holds everywhere on $\seq\bb$.
Let $\ES_{\bPhi}(n)$ denote the corresponding Ramsey
function, i.e. the smallest $N$ with the property above.
\end{definition}

For example, Theorem~\ref{t:es-monot} on monotone subsequences
can be re-stated as follows in this language: the
set $\bPhi=\{x_1<x_2,x_1\ge x_2\}$ of $1$-di\-mension\-al
predicates is {Erd\H{o}s--Szekeres},
with $\ES_\bPhi(n)=(n-1)^2-1$ (note that we do not write points
in $\R^1$, i.e., real numbers, in boldface---unlike points in $\R^d$). 

For Theorem~\ref{t:es-conv}
on subsets in convex position, we can let $\bPhi$ consist
of a single $4$-ary predicate $\Phi_{\rm conv}(\xx_1,\ldots,\xx_4)$
expressing that the 4-tuple $\xx_1,\ldots,\xx_4$ is in convex position
(with appropriate handling of collinearities). Indeed, a set is in convex
position if each of its 4-tuples is. Alternatively, we can set
$\bPhi':=\{\Phi_{\rm pos},\Phi_{\rm neg}\}$, where
$\Phi_{\rm pos}(\xx_1,\xx_2,\xx_3)$ expresses that the ordered
triple $(\xx_1,\xx_2,\xx_3)$ has a positive orientation, while
$\Phi_{\rm neg}$ expresses negative orientation. 

Let us remark that the notion of Erd\H{o}s--Szekeres set of predicates
is not general enough to express the colored Tverberg theorem mentioned
above, for example; however, it could easily be extended to a colored
setting, if that proved useful.\footnote{Perhaps more seriously,
there are results and problems in geometric Ramsey theory that do not
fit this kind of framework at all. For example, the famous recent
\emph{6-hole theorem} of Gerken \cite{gerken-6h} and Nicol\'as \cite{nicolas}
asserts that every sufficiently large point set $P\subset \R^2$
in general position contains an empty hexagon, i.e., six points
in convex position whose convex hull contains no other point of~$P$.
Here the problem is that this is not a ``pure'' Ramsey-type result,
since, although we pass to a subset, the original points still play
a role as ``obstacles''. We suspect general problems of this kind to be
substantially harder than those considered in this paper.
}

Another thing worth noting here
is that given a finite set $\bPhi$ of predicates, it is possible to
produce a single predicate $\Phi$ such that
$\bPhi$ is Erd\H{o}s--Szekeres iff $\{\Phi\}$ is.
Moreover, the functions $\ES_\bPhi$ and $\ES_{\{\Phi\}}$
can be related; see Section~\ref{s:replPhis} for a discussion.
However, the passage from $\bPhi$ to $\Phi$ increases
the number of variables and produces a rather cumbersome predicate $\Phi$,
while treating a set of predicates in our development is not
much more complicated than treating a single one, so
we prefer dealing with sets of predicates.

The predicates in Theorems~\ref{t:es-monot} and~\ref{t:es-conv}, 
as well as in the other examples
above and elsewhere in geometric Ramsey theory, can be represented
as semialgebraic predicates, where a \emph{$d$-di\-mension\-al
$k$-ary semialgebraic predicate} $\Phi=\Phi(\xx_1,\ldots,\xx_k)$
is a Boolean combination of polynomial equations and 
inequalities in the coordinates of the $k$ points $\xx_1,\ldots,\xx_k\in\R^d$
(the polynomials are assumed to have rational coefficients).
In this paper we consider only semialgebraic 
predicates.

We consider the following general questions quite natural and very
interesting.
\begin{enumerate}
\item (Ramsey function) \emph{What is the largest possible order
of magnitude of $\ES_{\bPhi}$, where $\bPhi$ is a finite
set of $d$-di\-mension\-al $k$-ary semialgebraic predicates?}

A simple upper bound follows from Ramsey's theorem.
For $\bPhi=\{\Phi_1,\ldots,\Phi_m\}$ and a sequence $\seq\aa$ of length $N$,
we color a $k$-tuple $I\subseteq \{1,2,\ldots,N\}$ with color 
$i\in\{1,2,\ldots,m\}$ if $\Phi_i$ holds
for  the $k$-tuple in $\seq\aa$ indexed by $I$ (if there are several possibilities, we pick one arbitrarily).
The color $m+1$ is used for the $k$-tuples on which no $\Phi_i$ holds.
By Ramsey's theorem, if $N\ge R_k(n;m+1)$ (the Ramsey number for $k$-tuples
with $m+1$ colors), $\seq\aa$ contains a subsequence
of length $n$ in which all $k$-tuples have
the same color. If $\bPhi$ is Erd\H{o}s--Szekeres, then for $n\ge \ES_\bPhi(k)$,
color $m+1$ is impossible. Thus, we get $\ES_\bPhi(n)\le
R_k(n;m+1)$ for all sufficiently large $n$, and $R_k(n;m+1)$ is bounded
above by a tower function of height~$k$---see, e.g., \cite{grs-rt-90}.
As shown by Alon et al.~\cite{AlonPPRS05}, if a semialgebraic
predicate $\Phi=\Phi(\xx_1,\xx_2)$
is binary and symmetric, i.e., 
$\Phi(\xx_1,\xx_2)\Leftrightarrow \Phi(\xx_2,\xx_1)$, 
and $\bPhi=\{\Phi,\neg\Phi\}$, then
$\ES_{\bPhi}$ is even bounded by a polynomial (depending on $d$ and~$\Phi$).
\item (Decidability) \emph{Is there an algorithm that,
given a finite set $\bPhi$ of $d$-di\-mension\-al semialgebraic predicates,
decides whether $\bPhi$ is Erd\H{o}s--Szekeres?}

By a celebrated result of Tarski \cite{t-dmeag-51}, the first-order
theory of the reals is decidable. That is, there is an algorithm
deciding the validity of formulas
of the form \[(Q_1x_1)\ldots(Q_k x_k)\Phi(x_1,\ldots,x_k),\] where
$Q_1,\ldots,Q_k\in\{\forall,\exists\}$ are quantifiers
for the real variables $x_1,\ldots,x_k$ and $\Phi$
is a semialgebraic predicate in our sense.  This is one of the most
useful decidability results, and there is an extensive literature
studying efficient algorithms for this problem and various
special cases (see, e.g., \cite{BasuPollackRoy-book}). 
This may give some hope for decidability of the Erd\H{o}s--Szekeres
property, and quantifier-elimination methods
for the first-order theory of the reals play a significant role
in the present paper.
\item\label{i:indiscer}
 (Homogeneous, or indiscernible, sequences) \emph{What can be said
about very long, or infinite, point sequences that are homogeneous
w.r.t.\ interesting classes of semialgebraic predicates?}

To explain this question, let us first consider all $2$-di\-mension\-al
semialgebraic predicates $\Phi$ that depend only on the orientation
of triples of points; for example, ``$\xx_1$ lies in the convex
hull of $\xx_2,\xx_3,\xx_4$''. Then every point sequence $C$ in convex position
(i.e., the points are in convex position and numbered along the
circumference of the convex hull either clockwise or counterclockwise)
is \emph{homogeneous} for this class of predicates, meaning that
each $k$-ary such predicate $\Phi$ either holds on all $k$-tuples
of $C$ or holds on none of them. Moreover, as far as these predicates
are concerned, all sequences in convex position look the same
(except for two possible orientations, that is).

However, some natural Ramsey-type questions, such as the colored Tverberg
problem or the question with triangles containing segment intersections
mentioned earlier, need ``higher-order'' geometric
predicates, such as whether the intersection of two lines,
each  determined by a pair of the considered points, lies above
or below the line spanned by another pair of points. (Semialgebraic
predicates of course include such predicates and much more.)
Here it is far from obvious what homogeneous sequences for such
predicates might look like. Let us remark that one kind of  homogeneous
set for semialgebraic predicates of bounded complexity,
called the \emph{stretched diagonal}, was used as
an interesting example in the recent papers \cite{BMN10,BMN11}, together with
a related construction of a \emph{stretched grid}---this is yet
another motivation for our investigations here.
\end{enumerate}

\heading{Answers for dimension~1. } 
Luckily, after investigating such problems for some time,
we found the (unpublished) thesis of Rosenthal
\cite{Rosenthal-thesis} which,
in the language of logic and model theory, provided an answer to the
third question (homogeneous sets) for $d=1$. With the help of
Rosenthal's results and methods, together with other techniques,
we then succeeded in answering the first two questions for $d=1$.

We were surprised by the answer to the first question:
it turned out that in dimension~1, the Ramsey functions
 can be at most doubly exponential,
independent of the arity of the predicates.

\begin{theorem}\label{t:rams-ub}
For every finite Erd\H{o}s--Szekeres set $\bPhi$ of
$1$-di\-mension\-al predicates
there exists a number $C=C(\bPhi)$ such that
$$
\ES_\bPhi(n) \le 2^{2^{Cn}}.
$$
\end{theorem}

The proof is presented in Sections~\ref{s:ra-seq}--\ref{s:rams-pf}.

In  Section~\ref{s:lwb} we will show that some predicates indeed require
a doubly exponential bound.

\begin{prop}\label{p:lwb} There is a finite 
Erd\H{o}s--Szekeres set $\bPhi$
of predicates with $\ES_\bPhi(n)\ge 2^{2^{cn}}$, with
a suitable constant $c>0$.
\end{prop}

The predicates in our proof of this proposition have arity~$5$.
After a preliminary version of the present paper appeared in the arXiv,
Conlon, Fox, Pach, Sudakov, and Suk \cite{cfpss-semialg} provided
a similar construction with arity 4, which is optimal,
since they also proved that every ternary semialgebraic predicate
has at most exponential Ramsey function.

The proof of Theorem~\ref{t:rams-ub} also immediately provides
a quantitative result for homogeneous subsequences for every 
semialgebraic predicate, Erd\H{o}s--Szekeres or not.

\begin{prop}\label{p:homog} For every finite set $\bPhi$
of $1$-dimensional semialgebraic 
predicates there is a constant $C=C(\bPhi)$ such that for every
integer $n$, every sequence $\seq a$ of length $2^{2^{Cn}}$
contains an $n$-term subsequence $\seq b$ such that each
predicate $\Phi\in\bPhi$ holds either everywhere on $\seq b$
or nowhere on~$\seq b$.
\end{prop}

By adding some more ingredients to the proof of Theorem~\ref{t:rams-ub}
in Sections~\ref{s:algo} and~\ref{s:feas}, we 
will also prove that $1$-di\-mension\-al 
Erd\H{o}s--Szekeres predicates can be recognized algorithmically.

\begin{theorem}\label{t:es1} 
There is an algorithm that, given a finite set $\bPhi$ of $1$-di\-mension\-al
semialgebraic predicates, decides whether it is Erd\H{o}s--Szekeres.
\end{theorem}

\heading{Higher dimensions: algebraic predicates
 and the multipartite case. } For $d\ge 2$, all of the three
 questions above appear much harder than
for $d=1$, and at the time of writing this paper we have
just some preliminary results for $d=2$.

Here we provide decidability results with $d$ arbitrary
for a restricted class of predicates, as well as for 
general semialgebraic predicates but with a different
Ramsey-type question.

Let us define an \emph{algebraic predicate}, a special
kind of  semialgebraic predicate, as 
a $d$-di\-mension\-al predicate
$\Phi(\xx_1,\ldots,\xx_k)$ expressible as a conjunction
of polynomial equations in the coordinates
of $\xx_1,\ldots,\xx_k$ (with rational coefficients).
We note that while a semialgebraic predicate defines a semialgebraic
set (in $(\R^d)^k$), an algebraic predicate defines an algebraic variety.

For algebraic predicates the question of being Erd\H{o}s--Szekeres
does not make much sense, since every nontrivial algebraic
predicate fails on a generic $k$-tuple of points.
However, the question of whether there exist arbitrarily long
point sequences, or even infinite ones, on which a given algebraic 
predicate holds everywhere, is meaningful. The following theorem
shows that it is decidable.

\begin{theorem}[``Effective compactness'' for algebraic predicates]
\label{t:algpred} For every $d$, $D$, and $k$ there exists
$N$, for which an explicit bound can be given,
 such that for every $d$-di\-mension\-al $k$-ary algebraic predicate
$\Phi$, in which all the polynomials have degree at most~$D$,
the following two conditions are equivalent:
\begin{enumerate}
\setlength{\itemsep}{1pt}
\setlength{\parskip}{0pt}
\setlength{\parsep}{0pt}
\item[\rm (i)] There exists a sequence $\seq\aa$ of $N$ points in
$\R^d$ on which $\Phi$ holds everywhere.
\item[\rm(ii)] There exists an infinite sequence $\seq\aa$
of points in $\R^d$ on which $\Phi$ holds everywhere.
\end{enumerate}
\end{theorem}

Since condition (i) can be tested using a decision algorithm
for the first-order theory of the reals, we obtain a decision
algorithm for testing the existence of infinite sequences,
or equivalently, of arbitrarily long sequences, on which
a given algebraic predicate holds everywhere.

Let us remark that the statement of the theorem also holds
with the real numbers replaced by the complex numbers.

We note that a statement analogous to Theorem~\ref{t:algpred}
fails badly for semialgebraic predicates. Indeed, the $1$-di\-mension\-al
binary predicate  $\Phi(x_1,x_2):=(x_1>0)\wedge (x_2\ge x_1+1)\wedge
(x_2\le 2x_1)$ holds everywhere on arbitrarily long finite sequences, 
such as $(n,n+1,\ldots,2n)$, but
on no infinite sequence.  Similarly, $(x_1>0)\wedge
(x_2\ge x_1+1)\wedge (x_2\le A)$, where $A$ is a constant,
admits $A$-term sequences but not longer.

Next, we let $\Phi(\xx_1,\ldots,\xx_k)$ be a $d$-di\-mension\-al
semialgebraic predicate, but we ask a different question.
Let $A_1,\ldots,A_k$ be point \emph{sets} in $\R^d$.
We say that $\Phi$ \emph{holds everywhere on $A_1\times\cdots\times A_k$}
if $\Phi(\aa_1,\ldots,\aa_k)$ holds for every choice of points
$\aa_1\in A_1$,\ldots, $\aa_k\in A_k$. We have the following
analog of Theorem~\ref{t:algpred}.

\begin{theorem}[``Effective compactness'' for the multipartite setting]
\label{t:multip} For every $d$, $D$, and $k$ there exists
$N$, for which an explicit bound can be given,
 such that for every $d$-di\-mension\-al $k$-ary semialgebraic predicate
$\Phi$ that is a conjunction of polynomial equations and
inequalities and in which all the polynomials have degree at most~$D$,
the following two conditions are equivalent:
\begin{enumerate}
\setlength{\itemsep}{1pt}
\setlength{\parskip}{0pt}
\setlength{\parsep}{0pt}
\item[\rm (i)] There exist $N$-point sets $A_1,\ldots,A_k$ in
$\R^d$ such that $\Phi$ holds everywhere on $A_1\times\cdots\times A_k$.
\item[\rm(ii)] There exist infinite sets $X_1,\ldots,X_k$ in $\R^d$
such that $\Phi$ holds everywhere on $X_1\times\cdots\times X_k$.
\end{enumerate}
For an arbitrary $d$-di\-mension\-al $k$-ary semialgebraic predicate
$\Phi$ a similar statement holds, except that $N$ also depends
on the length of $\Phi$, again in a way that could be made
explicit.
\end{theorem}

Theorems~\ref{t:algpred} and~\ref{t:multip} are proved in 
Section~\ref{s:multip}; the proofs are similar and more or less
independent of the rest of the paper.

\section{Ramseying for fast-growing sequences}\label{s:ra-seq}

Here we work in the 1-di\-mension\-al setting,
i.e., with sequences of real numbers, and we
begin with preparations for the proofs of
Theorems~\ref{t:rams-ub} and~\ref{t:es1}.

\paragraph{\boldmath $R$-growing sequences. }
The first idea in our approach is that if $\Phi$ is
a 1-di\-mension\-al semialgebraic predicate and $\seq a$ is a sequence that grows
sufficiently fast (where the speed of growth 
is quantified with respect to $\Phi$), 
then either $\Phi$ holds everywhere on $\seq a$,
or $\Phi$ holds nowhere on $\seq a$ (by which we mean that 
the negation $\neg\Phi$ holds everywhere on $\seq a$).
Moreover, one can decide between these two possibilities
``syntactically'', just from the structure of $\Phi$, 
without any information about $\seq a$ (except for a
guarantee of the fast growth).

This will be expressed more precisely below, but first
we formalize ``growing sufficiently fast''.

\begin{definition}\label{d:Rgrow}
Let $R>2$ be a real number. 
We call a sequence $\seq a=(a_1,a_2,\ldots,a_n)$
\emph{$R$-growing} if $a_1\ge R$ and $a_{i+1}\ge a_i^R$,
$i=1,2,\ldots,n-1$.
\end{definition}


\begin{observation}\label{o:lex-sign}
For every $1$-di\-mension\-al semialgebraic
predicate $\Phi=\Phi(x_1,\ldots,x_k)$ there 
exists $R>1$ such that
either $\Phi$ holds everywhere on every $R$-growing 
sequence $\seq a$, or $\Phi$ holds nowhere on every such $\seq a$. 
 Moreover, it is easy to decide by inspection of $\Phi$ which of these
cases holds (without explicit knowledge of~$R$).
\end{observation}

\begin{proof}[Sketch of proof. ] First we note that
if $p(x_1,\ldots,x_k)$ is a polynomial and $R$ is sufficiently
large in terms of the degree and maximum absolute value of the
coefficients of $p$, then the sign of $p(a_{i_1},a_{i_2},\ldots,a_{i_k})$,
$i_1<\cdots<i_k$, is given by the sign of the coefficient
of the lexicographically largest monomial present in $p(x_1,\ldots,x_k)$
(where a monomial $x_1^{\alpha_1}\cdots x_k^{\alpha_k}$
is lexicographically larger than $x_1^{\beta_1}\cdots x_k^{\beta_k}$
if the vector $(\alpha_k,\ldots,\alpha_1)$ precedes
$(\beta_k,\ldots,\beta_1)$ in the usual lexicographic ordering).
 The validity of
$\Phi$ can then be resolved based on the signs of the polynomials
appearing in~it.
\end{proof}

\paragraph{A Ramsey-type result with \boldmath$R$-growing sequences. }
The next idea  in our approach is that \emph{every} 
sufficiently long sequence $\seq a$ ``contains''
a long $R$-growing sequence, with a suitable meaning of
``contains''. 

Of course, not all long sequences contain $R$-growing subsequences;
at the very least, we also have to consider reversals of $R$-growing
sequences. Perhaps less obviously, we have sequences like
$$
A,A+1,A+2,\ldots,A+N,
$$
where $A$ is a number larger than $N$,
which contain neither fast-growing subsequences nor their reversals.

In this case, we can find a \emph{translate} of an $R$-growing
sequence, of the form $A+b_1,A+b_2,\ldots,A+b_n$, as a subsequence. 
For a slightly more sophisticated example, we consider the sequence
$$
A+\frac B1,A+\frac B2,\ldots,A+\frac BN,
$$
with two large parameters $A,B$. Here we have a subsequence
obtained from an $R$-growing sequence $(b_1,\ldots,b_n)$ by the
rational transformation $x\mapsto A+\frac Bx$.

A key insight, coming from the thesis of Rosenthal \cite{Rosenthal-thesis},
is that it suffices to consider rational transformations
of simple form and with at most two parameters (like $A$ and $B$ above).
To state this formally, we introduce the next definition.

\begin{definition} 
A \emph{$t$-parametric transformation} is a rational
function $f=f(x,X_1,\ldots,X_t)$ in $t+1$ variables.

Let $\seq b$ be a sequence of length $n$, let $f$ 
be a $t$-parametric transformation, and let $A_1,\ldots,A_t$
be real numbers. We write $f(\seq b,A_1,\ldots,A_t)$
for the sequence whose $i$th term is 
$f(b_i,A_1,\ldots,A_t)$ for all $i=1,2,\ldots,n$.

If $\seq a$ is a sequence of length $N$, we say that
$\seq b$ as above has a \emph{$t$-parametric embedding 
into $\seq a$ via $f$} if, for some $A_1,\ldots,A_t\in\R$,
$f(\seq b,A_1,\ldots,A_t)$  is a subsequence of~$\seq a$.

Finally, if $\FF$ is a set of $t$-parametric transformations,
we say that $\seq b$ has a \emph{$t$-parametric $\FF$-embedding}
into $\seq a$ if it has a $t$-parametric embedding
into $\seq a$ via some $f\in\FF$.
\end{definition}

Now we can state a key Ramsey-type result.
Let $\FF_0$ stand for the set of the
following two $2$-parametric transformations:
$$
f_1(x,X,Y):= X+Yx,\ \ \ \  f_2(x,X,Y):=X+\frac Yx.
$$

\begin{prop}\label{p:rams}
For every $n$ and $R$ there exists $N$ such that
for every sequence $\seq a$ of length $N$ there is an $R$-growing
sequence $\seq b$ of length $n$ such that $\seq b$ or $\rev (\seq b)$,
the reversal of $\seq b$, has a $2$-parametric $\FF_0$-embedding 
into~$\seq a$. Quantitatively, it suffices to take
\[
N:= 2^{R^{2n}},
\]
provided that $R\ge R_0$ for a sufficiently large constant~$R_0$.
\end{prop}

The proof is presented in Section~\ref{s:rams-pf} below.
A similar result, without the quantitative bound and stated
in a different language, is implicit in Rosenthal~\cite{Rosenthal-thesis}.

\section{Deciding predicates on transformed \boldmath$R$-growing sequences}
\label{s:pfT2}

Next, we would like to generalize Observation~\ref{o:lex-sign}
from an $R$-growing sequence to a sequence $\seq c$ obtained from
an $R$-growing sequence $\seq b$ (or its reversal) by a 
$2$-parametric transformation. 

Thus, let $\seq c=f(\seq b,A,B)$, for some
$2$-parametric transformation $f(x,X,Y)$ and some $A,B\in\R$.
For a given polynomial $p(x_1,\ldots,x_k)$,
we would like to understand the sign of
$p(c_{i_1},\ldots,c_{i_k})$.

Let us substitute $x_i=f(y_i,X,Y)$ into $p(x_1,\ldots,x_k)$,
and write the resulting rational function as the quotient
of two polynomials (in $y_1,\ldots,y_k,X,Y$).
We consider the sign of each of them separately.

Thus, let $q(y_1,\ldots,y_k,X,Y)$ be one of these two polynomials;
we write it as a polynomial in $y_1,\ldots,y_k$ whose coefficients
are polynomials in $X,Y$:
$$
q(y_1,\ldots,y_k,X,Y)=\sum_{\alpha\in\Lambda}
 q_\alpha(X,Y) y_1^{\alpha_1}\cdots y_k^{\alpha_k},
$$
where $\Lambda=\Lambda(q)$ is the set of all multiindices
$\alpha=(\alpha_1,\ldots,\alpha_k)$ such that the coefficient
$q_\alpha(X,Y)$ of $y_1^{\alpha_1}\cdots y_k^{\alpha_k}$
is not identically zero.

Unlike in Observation~\ref{o:lex-sign}, here we cannot in general
determine the sign of $q(b_1,\ldots,b_k,A,B)$ just from the knowledge
of the polynomial $q$ and from the fact that $\seq b$ is $R$-growing,
since we do not know the signs and order of magnitude of
the coefficients $q_\alpha(A,B)$.

Let us define, for $\alpha,\beta\in\Lambda$,
$$
\rho_{\alpha\beta}:= \frac{q_\alpha(A,B)}{q_\beta(A,B)}
$$
(where, for simpler notation, we put $\rho_{\alpha\beta}:=\infty$
if $q_\beta(A,B)=0$).
We observe that if each $\rho_{\alpha\beta}$  is either 
considerably smaller than $b_1$ or much larger than 
$b_k$, then the sign of $q(b_1,\ldots,b_k,A,B)$ can again
be determined. The following definition captures this condition.

\begin{definition}\label{d:type}
 Let $q=q(y_1,\ldots,y_k,X,Y)$ be a polynomial
and let $A,B\in \R$. We say that an $R$-growing sequence
$\seq b=(b_1,\ldots,b_n)$ is \emph{$R$-well-placed} w.r.t.\ $q,A,B$
if, for every $\alpha,\beta\in\Lambda(q)$, either
\begin{itemize}
\item $\rho_{\alpha\beta}$ is \emph{dwarfed} by $\seq b$,
meaning that $|\rho_{\alpha\beta}|\le b_1/R$, or
\item $\rho_{\alpha\beta}$ is \emph{gigantic} for $\seq b$,
meaning that $|\rho_{\alpha\beta}|\ge b_n^R$.
\end{itemize}
In this situation we define the \emph{type} of $q$
w.r.t. $A,B$, and $\seq b$ as the pair
$(\sigma,\tau)$, where $\sigma\:\Lambda\to \{-1,0,+1\}$
is given by $\sigma(\alpha):=\sgn q_\alpha(A,B)$, and
$\tau\:\Lambda^2\to \{{\rm D},{\rm G}\}$ is given by
$$
\tau(\alpha,\beta):=\alterdef{{\rm D}&\mbox{if 
$\rho_{\alpha\beta}$ is dwarfed by $\seq b$,}\\
{\rm G}&\mbox{if 
$\rho_{\alpha\beta}$ is gigantic for $\seq b$.}}
$$
\end{definition}

We also want to extend these notions from a single
polynomial $q$ to a collection of polynomials coming from
a predicate (or set of predicates).

Let $\bPhi$ be a (finite) set of semialgebraic predicates and
let $f=f(x,X,Y)$ be a $2$-parametric transformation.
For every polynomial $p(x_1,\ldots,x_k)$
occurring in some predicate of $\bPhi$ we proceed as above, i.e.,
we perform the substitution $x_i=f(y_i,X,Y)$, and we
express the resulting rational function as a quotient of
two polynomials in $y_1,\ldots,y_k,X,Y$.
Let $Q=Q(\bPhi,f)$ stand for the collection of all the polynomials
obtained in this way.

Then we will talk about an $R$-growing sequence $\seq b$
being \emph{$R$-well-placed} w.r.t.\ $Q,A,B$ (meaning that it is
$R$-well-placed w.r.t.\ $q,A,B$ for every $q\in Q$), and about the
\emph{type} of $Q$ w.r.t.\ $A,B$, and $\seq b$ (this is the
$|Q|$-tuple $\bigl((\sigma_q,\tau_q):q\in Q\bigr)$, where
$(\sigma_q,\tau_q)$ is the type of $q$ w.r.t.\ $A,B$, and $\seq b$).

Next, we extend our Ramsey-type result (Proposition~\ref{p:rams})
so that it yields well-placed sequences.

\begin{corollary}\label{c:rams}
Let $\FF_0$ be as in Proposition~\ref{p:rams}, and let $\bPhi$
be a set of semialgebraic predicates. Then there
is a constant $C_1=C_1(\bPhi)$ such that for every $n,R$ there 
exists 
\[N\le 2^{R^{ C_1n}}\]
such that for every sequence $\seq a$ of length $N$ there is
an $R$-growing sequence $\seq b$ of length $n$
such that $f(\seq b,A,B)$ or $f(\rev(\seq b),A,B)$
is a subsequence of $\seq a$ for some $f\in\FF_0$ and $A,B\in\R$,
and moreover, $\seq b$ is $R$-well-placed w.r.t.\ $Q(\bPhi,f),A,B$.
\end{corollary}

\begin{proof} Let us take $n':=C(n+2)$, where $C$ is a sufficiently
large number depending on $\Phi$ and $\FF_0$, and use Proposition~\ref{p:rams}
with $n'$ instead of $n$. Given a sufficiently long sequence  $\seq a$,
we find an $R$-growing sequence $\seq b$ of length $n'$
as in the conclusion of Proposition~\ref{p:rams}, such that
$f(\seq b,A,B)$ or $f(\rev(\seq b),A,B)$ is a subsequence
of~$\seq a$.

Now $\seq b$ is not necessarily $R$-well-placed w.r.t.\ 
$Q(\bPhi,f)$, $A,B$, and
so we consider all the values $|\rho_{\alpha\beta}|$ as in
Definition~\ref{d:type}, generated from all the polynomials in
$Q(\bPhi,f)$. These divide the real axis into intervals,
whose number can be assumed to be at most $C$.
Among these intervals, we fix one containing at least
$n'/C=n+2$ terms of $\seq b$. We take the (contiguous) subsequence of $\seq b$
contained in this interval, and we delete the first and last terms.
Since $\seq b$ is $R$-growing, it follows that
the remaining $n$-term subsequence is $R$-well-placed w.r.t.\ $Q(\bPhi,f),A,B$.
\end{proof}

Here is an analog of Observation~\ref{o:lex-sign}.

\begin{lemma}\label{l:type-dec}
Let $\Phi$ be a predicate, let $\seq b$ be an $R$-growing
sequence, and let $\seq c=f(\seq b,A,B)$ for some
$2$-parametric transformation $f$. Suppose that
$\seq b$ is $R$-well-placed w.r.t.\ $Q(\Phi,f),A,B$,
and that $R$ is sufficiently large in terms of $\Phi$ and~$f$.
Then $\Phi$ holds either everywhere on $\seq c$
or nowhere on it, and these possibilities can
be distinguished based on $\Phi$, $f$, and the type
of $Q(\Phi,f)$ w.r.t.\ $A,B,\seq b$ (without
knowing $A,B$ and~$\seq b$). A similar statement holds
for the validity of $\Phi$ on $\rev(\seq c)$.
\end{lemma}

\begin{proof}[Sketch of proof.]
It suffices to check that if $q=q(y_1,\ldots,y_k,X,Y)\in Q(\Phi,f)$ 
is a single polynomial, then $\sgn q(b_{i_1},\ldots,b_{i_k},A,B)$ is
the same for all choices of $i_1<\cdots<i_k$ (and it can be
deduced from the type of $q$).

The type of $q$ w.r.t.\ $A,B,\seq b$ gives us the signs
 of the terms $q_\alpha(A,B)b_{i_1}^{\alpha_1}\cdots b_{i_k}^{\alpha_k}$, $\alpha\in\Lambda(q)$.
We just need to check that whenever $\alpha,\beta\in\Lambda(q)$,
$\alpha\ne\beta$, the absolute values of the terms
$q_\alpha(A,B)b_{i_1}^{\alpha_1}\cdots b_{i_k}^{\alpha_k}$ and
$q_\beta(A,B)b_{i_1}^{\beta_1}\cdots b_{i_k}^{\beta_k}$ have
different orders of magnitude. If one of $q_\alpha(A,B)$
and $q_\beta(A,B)$ is $0$, we are done. If they are both nonzero
and $\rho_{\alpha\beta}$ is gigantic, the first term wins;
if $\rho_{\beta\alpha}$ is gigantic, the second term wins;
and if both $\rho_{\alpha\beta}$ and $\rho_{\beta\alpha}$ are dwarfed,
then the comparison is lexicographic according to $\alpha$ and $\beta$,
as in Observation~\ref{o:lex-sign}.
\end{proof}

\begin{proof}[Proof of Proposition~\ref{p:homog}.]
This proposition is an immediate consequence of Corollary~\ref{c:rams}
and Lemma~\ref{l:type-dec}.
\end{proof}

\begin{proof}[Proof of Theorem~\ref{t:rams-ub}.]
Let $k$ be the maximum arity of a predicate in $\bPhi$,
and let $n_0:=\ES_\bPhi(k)$ (this is well defined since
$\bPhi$ is Erd\H{o}s--Szekeres).

Let us consider $n\ge n_0$, let $N:=2^{R^{C_1n}}$ be as in
Corollary~\ref{c:rams}, and let $\seq a$ be
an arbitrary sequence of length $N$. Corollary~\ref{c:rams} yields
an $R$-growing sequence $\seq b$ of length $n$ 
such that, for some $f\in\FF_0$ and $A,B\in\R$,
$\seq c:=f(\seq b,A,B)$ or $\seq c:=f(\rev(\seq b),A,B)$ is a subsequence
of $\seq a$. Moreover, $\seq b$ is $R$-well-placed w.r.t.\ $Q(\bPhi,f),A,B$.
Then by Lemma~\ref{l:type-dec}, each $\Phi\in\bPhi$ holds either
everywhere on $\seq c$ or nowhere on $\seq c$.

Since $n\ge n_0=\ES_\bPhi(k)$, the sequence $\seq c$ contains at least
one $k$-tuple on which some $\Phi\in\bPhi$ holds, and hence
this $\Phi$ holds everywhere on $\seq c$. This shows that
$\ES_\bPhi(n)\le 2^{R^{C_1n}}$ for all $n\ge n_0$.
Since $R$ depends only on $\bPhi$, this gives the 
desired bound $\ES_\bPhi(n)\le 2^{2^{O(n)}}$
for all $n\ge n_0$, and the finitely many $n<n_0$
can be taken care of by setting $C$ sufficiently large.
\end{proof}

The way we obtained the existence of $n_0$ in the proof above
was nonconstructive. Alternatively, one can obtain an explicit
dependence of $n_0$ on $\bPhi$ using the idea in Section~\ref{s:feas}.

\section{Proof of the Ramsey-type result for sequences}
\label{s:rams-pf}

We begin with a quantitative version of the ``dichotomy lemma''
of Rosenthal~\cite{Rosenthal-thesis}.
Let us say that a sequence $\seq b=(b_1,\ldots,b_n)$
satisfies the \emph{doubling differences condition},
or the DDC for short,
if $|b_k-b_i|\ge 2|b_j-b_i|$ holds for every~$i<j<k$
(or, in our previous terminology, the predicate ``$|x_3-x_1|\ge 2|x_2-x_1|$''
holds everywhere on $\seq b$).
We note that the differences in a sequence $\seq b$ satisfying
the DDC grow (at least) exponentially,
e.g., $|b_i-b_1|\ge 2^{i-1}|b_2-b_1|$. 

\begin{lemma}\label{l:doubl}
 Let $n$ be a natural number, and let $N={2n\choose n}\le 4^n$.
Then every (strictly) increasing sequence $\seq a$ 
of real numbers of length $N$
has a subsequence $\seq b$ of length $n$ such that
one of the two (increasing) sequences $\seq b$
and $\rev(-\seq b)$ (where $-\seq b$ stands for $(-b_1,\ldots,-b_n)$)
 satisfies  the DDC.
\end{lemma}

\begin{proof} 
Before proceeding, let us note that if we are not interested in
the quantitative bound, the existence of a suitable $N$ follows
easily from Ramsey's theorem for triples. Indeed, we observe
that, assuming $b_i< b_j< b_k$,  the DDC
 $b_k-b_i\ge 2(b_j-b_i)$ is equivalent to $b_k-b_j\ge b_j-b_i$.
Given an increasing sequence $(a_1,a_2,\ldots,a_N)$, we color
a triple $\{i,j,k\}\subseteq [N]$, $i<j<k$, red if 
$a_k-a_j\ge a_j-a_i$, and blue otherwise. Now a red homogeneous
subset corresponds to a subsequence $\seq b$ satisfying the DDC, 
 and a blue homogeneous subset corresponds to
a subsequence $\seq b$ with $\rev(\seq b)$ satisfying the DDC. 
The latter sequence
is decreasing rather than increasing, but we can repair this
by considering $\rev(-\seq b)$, since  negation preserves the~DDC. 

Now we present the proof giving a better quantitative bound;
it resembles one of the well-known proofs of the Ramsey theorem for 
graphs. Let us define $N(k,\ell)$ as the smallest $N$ such
that every increasing sequence $\seq a$ of length $N$
contains a subsequence $\seq b$ of length $k$ satisfying the DDC,
or a subsequence $\seq b$ of length $\ell$ with $\rev(\seq b)$
satisfying the DDC.

We have the initial conditions $N(2,\ell)=2$ and $N(k,2)=2$, and
below we will derive the recurrence
\[
N(k,\ell)\le N(k-1,\ell)+N(k,\ell-1)-1,\ \ k,\ell\ge 3.
\]
It is well known, and easy to check, that this implies
$N(k,\ell)\le {k+\ell\choose k}$, and so the bound in the lemma follows.

To verify the recurrence, let $N:=N(k-1,\ell)+N(k,\ell-1)-1$
and let $\seq a$ be a nondecreasing sequence of length $N$.
We divide the interval $[a_1,a_N]$ into the left and right
subinterval by the midpoint $\frac12(a_1+a_N)$, and we
let $\seq a'$ and $\seq a''$ be the parts of $\seq a$
in the left and right subinterval, respectively.
Then  $\len \seq a'\ge N(k-1,\ell)$ 
or $\len \seq a''\ge N(k,\ell-1)$ (where $\len\seq a$ denotes the
length of a sequence~$a$).

In the former case, if $\seq a'$ contains a subsequence
$\seq b$ of length $\ell$ with $\rev(\seq b)$ satisfying the
DDC, we are done. Otherwise, $\seq a'$ has 
a subsequence $\seq b'$ of length $k-1$ satisfying the
DDC, and it is easy to see that the sequence
obtained by appending $a_N$ to $\seq b'$ also satisfies the DDC.
The case of $\len \seq a''\ge N(k,\ell-1)$ is analogous.
This concludes the proof.
\end{proof}

Rather than using Lemma~\ref{l:doubl} directly, we will apply 
the following simple consequence, where the DDC is strengthened
to $R$-fold expansion of the differences.

\begin{corollary}\label{c:R-fold}
Let $n$ be a natural number, let $R>1$, and set
$r:=\lceil\log_2 R\rceil$. Then there is $N\le  4^{r(n-1)+1}$
such that every (strictly) increasing sequence $\seq a$
of real numbers of length $N$
has a subsequence $\seq b$ of length $n$ such that
the predicate ``$x_3-x_1\ge R(x_2-x_1)$'' holds everywhere on
one of the two (increasing) sequences $\seq b$
and $\rev(-\seq b)$.
\end{corollary}

\begin{proof} Select a subsequence $\seq b'$ of $\seq a$ of length $r(n-1)+1$
as in Lemma~\ref{l:doubl}, and define $b_i:= b'_{r(i-1)+1}$,
$i=1,2,\ldots,n$.
\end{proof}

\begin{proof}[Proof of Proposition~\ref{p:rams}.]
Let $n$ and $R$ be given, let $N$ be sufficiently large
as in the proposition, and let $\seq a$ be a sequence of length~$N$.

First, if some number occurs at least $n$ times in $\seq a$,
we are done, since the $1$-parametric transformation $f(x,X)=X$
(a special case of both $f_1$ and $f_2$ with $Y=0$)
embeds any $n$-term sequence into $\seq a$. Otherwise,
we may pass to a subsequence $\seq a'$ with all terms
distinct and of length at least $N/n$. Next, by the
Erd\H{o}s--Szekeres lemma, we can further pass to
a subsequence $\seq a''$, with $\len (\seq a'')\ge
\sqrt{N/n}$, that is either increasing or decreasing.
We consider only the increasing case, the decreasing one
being symmetric.

Next, applying Corollary~\ref{c:R-fold}, we obtain a subsequence
$\seq a^{(3)}$ of $\seq a''$, where $\len \seq a''\le 4^{r(\len\seq a^{(3)}-1)+1}$, such that 
``$x_3-x_1\ge R(x_2-x_1)$'' holds everywhere on $\seq a^{(3)}$
or $\rev(-\seq a^{(3)})$.
Let us again discuss only the former case.

In this case we form a new sequence $\seq b^{(3)}$ of length 
$\len\seq a^{(3)}-1$ by setting $b^{(3)}_i:= a^{(3)}_{i+1}-a^{(3)}_1$. 
We note that $\seq b^{(3)}$ satisfies $b^{(3)}_1>0$ 
and $b^{(3)}_{i+1}\ge R b^{(3)}_{i}$,
and it embeds into $\seq a$ via the $1$-parametric transformation
$f(x,X)=x+X$ for $X:=a^{(3)}_1$.

Now we apply Corollary~\ref{c:R-fold} again, this time to the sequence
$\seq \ell$ with $\ell_i=\log b^{(3)}_i$, and this yields
a subsequence $\seq b^{(4)}$ of $\seq b^{(3)}$, such that
$$
\frac {b^{(4)}_k}{b^{(4)}_i} \ge
\left(\frac {b^{(4)}_j}{b^{(4)}_i}\right)^{\!\!R} 
$$
holds for every $i<j<k$, or a similar relation holds for 
$\rev(1/\seq b^{(4)})$.
As expected, we again deal explicitly only with the
first possibility, and we set $c^{(4)}_i := b^{(4)}_{i+1}/b^{(4)}_1$,
$i=1,2,\ldots,\len (\seq b^{(4)})-1$.

By this choice, we have $c^{(4)}_{i+1}\ge (c^{(4)}_{i})^R$.
Moreover, since $b^{(4)}_{i+1}\ge Rb^{(4)}_{i}$ for all $i$,
we also obtain $c^{(4)}_1=b^{(4)}_2/b^{(4)}_1\ge R$,
and so $\seq c^{(4)}$ is $R$-growing. It is also clear
that $\seq c^{(4)}$ embeds into $\seq b^{(3)}$ via the 1-parametric
transformation $g(x,Y):= x Y$ with $Y:= b^{(4)}_1$, and hence
it embeds into $\seq a$ via the $2$-parametric transformation
$h(x,X,Y)=f(g(x,X),Y)$. Thus, $\seq c^{(4)}$ is our desired
sequence (called $\seq b$ in the proposition).

By following the chain of length estimates backwards, we
get that 
\[N\ge n4^{2(r(4^{r(n-1)+2}-1)+2)} \] 
suffices for the whole argument. Assuming that $R$ exceeds
a suitable constant (not very large, actually), a series
of rough estimates shows that the right-hand side
can be bounded by $2^{R^{2n}}$ as claimed.
\end{proof}

\section{The decision algorithm}\label{s:algo}

Before formulating the algorithm for deciding whether
a given set of predicates $\bPhi$ is Erd\H{o}s--Szekeres, we
introduce some additional terminology.

If $q=q(y_1,\ldots,y_k,X,Y)$ is a polynomial,
a \emph{candidate type} for $q$ is a pair $(\sigma,\tau)$
of arbitrary functions $\sigma\:\Lambda\to\{-1,0,+1\}$
and $\tau\:\Lambda^2\to\{{\rm D},{\rm G}\}$, where
$\Lambda$ is as introduced above in Definition~\ref{d:type}.
A \emph{candidate type} for a collection $Q$
of polynomials (all in the variables $y_1,\ldots,y_k,X,Y$)
is an arbitrary sequence $T=\bigl((\sigma_q,\tau_q):q\in Q\bigr)$,
where each $(\sigma_q,\tau_q)$ is a candidate type for~$q$.

\begin{definition}
\label{d:feas}
We call a candidate type $T$ for a collection $Q$ of polynomials
\emph{feasible} if for every $R$ and every $n$ there exist 
$A,B\in\R$ and an $R$-growing sequence 
$\seq b$ of length $n$ that is $R$-well-placed w.r.t.\ $Q$, such that
the type of $Q$ w.r.t.\ $A,B$, and $\seq b$ equals~$T$.
Any such sequence $\seq b$ is called a \emph{feasible $R$-growing
sequence} for $Q$ and~$T$.
\end{definition}

We note that our definition of feasibility does not use any
particular value of $R$, but has the quantification ``for all $R$''.
Accordingly, we will never use an explicit value of $R$ in our
algorithm.

In Section~\ref{s:feas} below, we will provide a subroutine
that tests feasibility of a given candidate type
for a given collection~$Q$ of polynomials. 
Now we present the main algorithm, in which we use this feasibility
testing as a black box.

\paragraph{Algorithm for testing if $\bPhi$ is Erd\H{o}s--Szekeres. }\ \\
\noindent\emph{Input: } A finite set $\bPhi$ of semialgebraic predicates.\\ 
\noindent\emph{Output: } YES if $\bPhi$ is Erd\H{o}s--Szekeres, NO otherwise.
\begin{enumerate}
\item\label{s::} Let $\FF_0$ be the set of two $2$-parametric transformations
as in Proposition~\ref{p:rams}. For each $f\in\FF_0$ let $Q(\bPhi,f)$ be the 
collection of polynomials obtained from $\bPhi$ (as introduced after 
Definition~\ref{d:type}).
Perform the following steps
for every $f\in\FF_0$ and every candidate type $T$ for $Q(\bPhi,f)$. 
If all these steps are completed without
returning NO, return YES and finish.
\item\label{s:}  
Test the feasibility of $T$ (as specified in Section~\ref{s:feas} below).
If $T$ is feasible, continue with the next step; otherwise, continue
at Step~\ref{s::} with the next $T$ or next~$f$.
\item\label{s:::}
 Determine, by the method of Lemma~\ref{l:type-dec},
whether all $\neg\Phi$ for $\Phi\in\bPhi$, hold everywhere on $f(\seq b,A,B)$
or on $f(\rev(\seq b),A,B)$,
where $\seq b$ is a feasible $R$-growing sequence for $Q(\bPhi,f)$
and $T$ (here $A,B$ are the corresponding parameter
values, which we need not determine explicitly;
similarly, $R$ is only assumed to be sufficiently large).
If they do, return NO
and finish the whole algorithm. If not, continue at Step~\ref{s::}
with the next $T$ or next~$f$.
\end{enumerate}

\begin{proof}[Proof of Theorem~\ref{t:es1} (assuming the feasibility testing).] 
The algorithm is clearly finite, so it suffices to verify that
its answer is always correct. With the tools developed above,
the proof is routine.

First suppose that the algorithm
 returns NO; then this is because in Step~\ref{s:::} it has 
found arbitrarily long sequences $\seq c$ on which all $\neg\Phi$,
$\Phi\in\bPhi$, hold everywhere, namely, the sequences 
$f(\seq b,A,B)$ or $f(\rev(\seq b),A,B)$ with $\seq b$
a feasible $R$-growing sequence for $Q(\bPhi,f)$ and the considered 
type $T$, where $R$ is sufficiently large (depending on $\bPhi$),
$A$ and $B$ are suitable values of the parameters,
and $\seq b$ can be chosen as long as desired.
So $\bPhi$ is not Erd\H{o}s--Szekeres.
\medskip

Now suppose that $\bPhi$ is not Erd\H{o}s--Szekeres; this means that
for some $n_0$, there exist arbitrarily long sequences for which
no $\Phi\in\bPhi$ holds everywhere on any subsequence of length $n_0$.
By Ramsey's theorem, this means that there are arbitrarily long
sequences $\seq a$ on which each of $\neg\Phi$, $\Phi\in\bPhi$,
holds everywhere. (Alternatively, we can avoid
using Ramsey's theorem and argue as in the proof of Theorem~\ref{t:rams-ub}
at the end of Section~\ref{s:pfT2}.)
For every $N$, fix such a sequence $\seq a^{(N)}$ of length~$N$.

For every $n$, by Corollary~\ref{c:rams},
there is $N=N(n)$ such that for some $n$-growing $n$-term
sequence $\seq b^{(n)}$, some $f^{(n)}\in\FF_0$, and
some $A^{(n)},B^{(n)}\in\R$, one of the sequences $f(\seq b^{(n)},
A^{(n)},B^{(n)})$ and $f(\rev(\seq b^{(n)}),
A^{(n)},B^{(n)})$ is a subsequence of $\seq a^{(N(n))}$,
and $\seq b^{(n)}$ is $n$-well-placed 
w.r.t.\ $Q(\Phi,f^{(n)})$, $A^{(n)},B^{(n)}$.
Let $T^{(n)}$ be the type of $Q(\bPhi,f^{(n)})$ w.r.t.\ $A^{(n)},B^{(n)}$,
and $\seq b^{(n)}$.

Since $\FF_0$ is finite and there are finitely many types,
there is an infinite subsequence $n_1<n_2<\cdots$ such that
all the $f^{(n_j)}$ are equal to the same $f$ and all the 
$T^{(n_j)}$ are equal to the same~$T$. But then we get
that $T$, regarded as a candidate type for $Q(\bPhi,f)$,
is feasible, and Step~\ref{s:::} returns NO for this $f$
and~$T$. This concludes the proof.
\end{proof}

\section{Testing the feasibility of a candidate type}\label{s:feas}

The problem of testing feasibility of a given candidate type for a 
collection of polynomials can be recast in the following (slightly more
general) terms. We are given a finite collection $\QQ$ of 
bivariate polynomials (the $q_{\alpha}(X,Y)$ in the setting of 
Definition~\ref{d:feas}), and for each $q\in\QQ$, a sign
$\sigma_q\in\{-1,0,+1\}$ is specified. Moreover, we are given two
finite collections $\DD$ and $\GG$ of rational functions in the variables
$X,Y$ (in the setting of Definition~\ref{d:feas}, $\DD$ consists of
the $\rho_{\alpha\beta}$ that should be dwarfed by the desired 
feasible sequence, while $\GG$ are the $\rho_{\alpha\beta}$ that 
should be gigantic for it). The numerators and
denominators of the rational functions in $\DD$ and in $\GG$ belong
to $\QQ$, and thus their signs are under control. Moreover, we may assume
that the prescribed signs for all the denominators are nonzero.

The question then is whether we can make a large enough gap between
the values of the functions in $\DD$ and those in $\GG$ so that
an $n$-term $R$-growing sequence fits there. Formally, we thus ask
for the validity of the formula
$$
\Psi := \forall n\in\N\,\, \forall R\, \exists L,H
:(L\ge R)\wedge (H\ge L^{R^{n+2}})\wedge
\Xi(L,H),
$$
where
\begin{eqnarray*}
\Xi(L,H) &:=& \exists X,Y: \Bigl(\bigwedge_{q\in \QQ}\sgn q(X,Y)=\sigma_q\Bigr)
\\&&\ \ \ 
              \wedge \Bigl(\bigwedge_{\rho\in\DD} \rho(X,Y)\le L\Bigr)
              \wedge \Bigl(\bigwedge_{\rho\in\GG} \rho(X,Y)\ge H\Bigr). 
\end{eqnarray*}

We would like to test the validity of $\Psi$
using a decision algorithm for the first-order
theory of real-closed fields (the first such algorithm
is due to Tarski \cite{t-dmeag-51}, and we refer, e.g., to
\cite{BasuPollackRoy-book} for more recent ones). While $\Xi(L,H)$ can easily
be rewritten to a first-order formula in the theory of the reals
(since $\sgn q(X,Y)=\sigma_q$ is expressed as a polynomial inequality,
and the inequalities $\rho(X,Y)\le L$ or $\rho(X,Y)\ge H$
can be multiplied by the denominator, since we know its sign),
$\Psi$ is not such a formula: quantification over the
natural numbers, and more importantly, the exponential function,
do not belong to the first-order theory of the reals.

To remedy this, we replace $\Psi$ by the formula
$$
\Psi^*:= \forall R\,\,\exists L:
(L\ge R)\wedge \forall H\,\,\Xi(L,H).
$$
The validity of $\Psi^*$ can be tested by the decision algorithms
mentioned above, and so for completing our subroutine for
feasibility testing, it suffices to prove the following.

\begin{lemma}\label{l:equiv}
The formulas $\Psi$ and $\Psi^*$ are equivalent.
\end{lemma}

\begin{proof} Clearly $\Psi^*$ implies $\Psi$. For the reverse
implication, let us consider the function
$$
h(L):=\sup \{H:\Xi(L,H)\}.
$$
Supposing that $\Psi^*$ does not hold, we see that there are arbitrarily
large values of $L$ such that $h(L)$ is finite. 
Since $h$ is a semialgebraic function, it has to be finite
for all $L\ge L_0$ for some $L_0$. But if $\Psi$ did hold,
then for every $m$ there are arbitrarily large values of $L$
with $h(L)\ge L^m$, while 
a semialgebraic function may have at most a polynomial growth.
Indeed, Bochnak, Coste and Roy \cite[Proposition 2.6.2]{BoCoRo}
have such a statement with the additional assumption that the
considered semialgebraic function is continuous (apparently in
the interest of a simpler proof). To see this in the required
slightly greater generality, one may first note that a semialgebraic function
is piecewise algebraic, with finitely many pieces
(\cite[Lemma~2.6.3]{BoCoRo}), and check that an algebraic function $h$
has at most polynomial growth (for this, we take
a polynomial $g$ with $g(h(x),x)=0$ and consider the terms with the largest
power of the first variable).

Hence $\Psi$ does not hold and the lemma is proved.
\end{proof}

\section{A lower bound for the Ramsey function}\label{s:lwb}

\begin{proof}[Proof of Proposition~\ref{p:lwb}.]
 We recall that the
\emph{cross-ratio} of an ordered $4$-tuple $(z_1,\ldots,z_4)$
of real numbers is defined as
\[
(z_1,z_2;z_3,z_4):=\frac{(z_1-z_3)(z_2-z_4)}{(z_2-z_3)(z_1-z_4)}.
\]
As is well known, the cross-ratio is invariant under projective
transforms of the form $x\mapsto (ax+b)/(cx+d)$ (assuming $ad-bc\ne0$),
and in particular, it is invariant under the two 
transformations  $f_1(x,X,Y)=X+Yx$ and $f_2(x,X,Y)=X+Y/x$ of $\FF_0$ 
(unless $Y=0$, that is).

We define the set of predicates $\bPhi=\{\Phi_1,\Phi_2,\Phi_3\}$
as follows:
\begin{eqnarray*}
\Phi_1(x_1,x_2) &:=& x_1=x_2,\\
\Phi_2(x_1,\ldots,x_5)&:=&
 \mathrm{DISTINCT}(x_1,\ldots,x_5)\,\wedge\,
 (x_1,x_2;x_3,x_4)^2\ge 4 \\
&&\ \wedge\, (x_1,x_2;x_3,x_5)^2\ge (x_1,x_2;x_3,x_4)^4,\\
\Phi_3(x_1,\ldots,x_5)&:=& \Phi_2(x_5,x_4,\ldots,x_1).
\end{eqnarray*}
Here $\mathrm{DISTINCT}(x_1,\ldots,x_5)$ has the expected meaning.
The condition $(x_1,x_2;x_3,x_4)^2\ge 4$ should be read
as $|(x_1,x_2;x_3,x_4)|\ge 2$, but it is written in squared form,
so that it can be expressed as a polynomial inequality
and we need not worry about the sign. A similar comment
applies to the last condition of~$\Phi_2$.

First we check that $\bPhi$ is Erd\H{o}s--Szekeres.
Thus, let $n$ be given, let $R$ be a suitable
constant, and let $\seq a$ be a sufficiently long sequence.
By Proposition~\ref{p:rams}, there is an $n$-term $R$-growing sequence
$\seq b$ such that $\seq b$ or $\rev(\seq b)$ has a 2-parametric
embedding in $\seq a$ via $\FF_0$. 

It is routine, and somewhat tedious,
 to check that $\Phi_2$ holds everywhere on any
$R$-growing sequence with a sufficiently large constant $R$
(intuitively it is clear that the cross-ratios have to grow fast,
and formally it can be checked as in Observation~\ref{o:lex-sign}),
and consequently, $\Phi_3$ holds everywhere on the reversal
of an $R$-growing sequence.

Since $\Phi_2$ and $\Phi_3$ are expressed in terms of cross-ratios,
and these are invariant under the transformations of $\FF_0$,
we get that if $\seq b$ $\FF_0$-embeds into $\seq a$,
then either the image of $\seq b$ is a constant sequence, or 
then $\Phi_2$ holds everywhere on the image of $\seq b$.
 Similarly, if $\rev(\seq b)$ $\FF_0$-embeds into
$\seq a$, then either the corresponding subsequence of
$\seq a$ is constant, or $\Phi_3$ holds everywhere on it.
Thus, $\bPhi$ is indeed Erd\H{o}s--Szekeres.

Next, we want to prove a doubly exponential lower bound
for $\ES_\bPhi(n)$.
To this end, let $\seq a=(1,2,\ldots,N)$, and let $\seq c$ be an
$n$-term subsequence on which one of $\Phi_2,\Phi_3$
holds everywhere ($\Phi_1$ is out since the terms of $\seq a$
are all distinct). By possibly replacing $\seq c$ by
its reversal, we may assume that $\Phi_2$ holds everywhere
on $\seq c$. Then the cross-ratios in $\seq c$
satisfy $|(c_1,c_2;c_3,c_4)|\ge 2$ and
$|(c_1,c_2;c_3,c_{i+1})|\ge |(c_1,c_2;c_3,c_{i})|^2$,
and so $|(c_1,c_2;c_3,c_{n})|\ge 2^{2^{n-4}}$.
Since all terms of $\seq c$ are integers, at least
one of them has to be at least $2^{2^{cn}}$.
\end{proof}

\section{Replacing a set of predicates by a single predicate}\label{s:replPhis}

Here we consider two simple constructions for replacing a finite set
$\bPhi=\{\Phi_1,\ldots,\Phi_r\}$ of $d$-di\-mension\-al
$k$-ary predicates with a single predicate $\bar\Phi$, in such 
a way that $\bPhi$ is Erd\H{o}s--Szekeres iff $\bar\Phi$ is.

In the first construction, we simply set
 $\bar\Phi:=\bigvee_{i=1}^r\Phi_i$.
Then if some $\Phi_i$ holds everywhere on some sequence
$\seq\aa$, then so does $\bar\Phi$. Conversely, assuming that
$\bar\Phi$ holds everywhere on a sequence $\seq\aa$
of length $N$, we color
each $k$-tuple $j_1<\cdots<j_k$ of indices by a color $i$
such that $\Phi_i(\aa_{j_1},\ldots,\aa_{j_k})$ holds. Then by Ramsey's
theorem, some $\Phi_i$ holds everywhere on a suitable subsequence
of length $n$, provided that $N$ is sufficiently large in terms
of $n$, $k$, and~$r$. 

However, for this construction,
it might happen than $\ES_{\{\bar\Phi\}}$ is much larger than
$\ES_{\bPhi}$. Here is the second construction, which preserves
the Ramsey function but increases the number of variables and
the complexity of the predicate considerably.

\begin{lemma} Let $\bPhi=\{\Phi_1,\ldots,\Phi_r\}$ be a set of
$k$-ary predicates.
Then there is a predicate $\bar\Phi$ in $rk$ variables
such that $\{\bar\Phi\}$ is
Erd\H{o}s--Szekeres iff $\bPhi$ is, and 
if they are Erd\H{o}s--Szekeres, then
$\ES_{\bar \Phi}(n)=\ES_\bPhi(n)$ for all $n\ge rk$.
\end{lemma}

\begin{proof}
We first define $(rk)$-ary predicates $\Psi_i$, 
$i=1,2,\ldots,r$, where $\Psi_i(\xx_1,\ldots,\xx_{rk})$
expresses that $\Phi_i$ holds everywhere on the sequence
$(\xx_1,\xx_2,\ldots,\xx_{rk})$. Explicitly,
$$
\Psi_{i}(\xx_1,\ldots,\xx_{rk}):=\bigwedge_{1\le j_1<\cdots<j_k\le rk} 
\Phi_i(\xx_{j_1},
\ldots,\xx_{j_k}).
$$
Then we set $\bar\Phi(\xx_1,\ldots,\xx_{rk}):=
\bigvee_{i=1}^r \Psi_i(\xx_1,\ldots,\xx_{rk})$.

Clearly, if some $\Phi_i$ holds everywhere on a sequence
$\seq a$, then so does $\bar\Phi$. Conversely, suppose that
$\bar\Phi$ holds everywhere on a sequence $\seq a$ of length $n\ge rk$;
we claim that then some $\Phi_i$ holds everywhere on $\seq a$ as well.

If it were not the case, we fix, for every $i=1,2,\ldots,r$,
a $k$-tuple $J_i\subseteq [n]$ such that $\Phi_i$ does not hold
on the corresponding $k$ terms of $\seq\aa$. We consider the union
$\bigcup_{i=1}^r J_i$ and, if it has fewer than $rk$ elements,
we add the appropriate number of other elements of $[n]$ 
(chosen arbitrarily) to it, obtaining an $rk$-element set $J$.
Then $\bar\Phi$ does not hold on the subsequence of $\seq\aa$ indexed by $J$;
the resulting contradiction proves the lemma.
\end{proof}

\section{Algebraic predicates, and the multipartite setting}
\label{s:multip}

In this section we prove the effective compactness
for  $d$-dimensional algebraic predicates (Theorem~\ref{t:algpred}),
as well as the effective compactness for semialgebraic predicates
in the multipartite setting (Theorem~\ref{t:multip}).
We begin with the multipartite setting, since the proof is formally somewhat
simpler.

We will repeatedly use the following straightforward fact:
for every set $P\subseteq \R[x_1,\ldots,x_t]$ of 
$t$-variate polynomials, each of degree at most $D$,
there is a subset $P_0\subseteq P$ of at most ${D+t\choose t}$
polynomials that defines the same variety in $\R^t$ as $P$,
where the variety defined by $P$ is $V(P)=
\{\xx\in\R^t: p(\xx)=0\mbox{ for all }p\in P\}$.
Indeed, the vector space of all $t$-variate polynomials
of degree at most $D$ has dimension ${D+t\choose t}$,
with the set of all monomials of degree at most $D$ forming a basis, and if
we choose $P_0$ as a basis of the subspace generated
by $P$, then we have $V(P_0)=V(P)$.
We will refer to this fact as the \emph{bounded-dimension
argument}.

\begin{proof}[Proof of Theorem~\ref{t:multip}. ]
Throughout the proof, by saying that a certain quantity
is \emph{bounded} we mean that
it can be bounded from above by some explicit function
of $d$ (the space dimension), $k$ (the arity of the predicate), and $D$
(the maximum degree of the polynomials in the predicate).

To explain the idea of the proof
in a simpler setting, we first assume that
$\Phi$ is \emph{algebraic} and binary, i.e., $k=2$.

We need to prove only the implication (i)\,$\Rightarrow$\,(ii).
So we assume that $\Phi(\xx_1,\xx_2)$ holds everywhere
on $A_1\times A_2$, where $|A_1|=|A_2|=N$ is large but so
far unspecified. 
Our goal is to construct infinite 
sets $X_1,X_2\subseteq \R^d$ such that $\Phi$ holds everywhere
on $X_1\times X_2$.

By the assumptions of the theorem, 
 $\Phi$ is a conjunction of polynomial equations
of degree at most $D$, and by the bounded-dimension argument,
we can assume that the number of equations is bounded.

The plan is the following: We are going to define two
semialgebraic sets (and actually, algebraic varieties)
$V_{1,2},V_{2,1}\subseteq \R^d$, such that 
$A_1\subseteq V_{1,2}$, $A_2\subseteq V_{2,1}$, and
$\Phi$ is easily checked to hold on $V_{1,2}\times V_{2,1}$. 
We will also show that
$V_{1,2}$ and $V_{2,1}$ can be defined by bounded-size formulas 
 (i.e., the formula is a Boolean combination
of at most $m$ polynomial equations and inequalities,
each of degree at most $\bar D$, where $m$ and $\bar D$ are bounded).
Then we will use a result stating that the number of connected
components of a semialgebraic set definable by a bounded-size
formula is bounded; see \cite{GabriVor} or
\cite[Theorem~7.50]{BasuPollackRoy-book}.
Finally, if $N=|A_1|$ is larger than the number of components
of $V_{1,2}$, then two points of $A_1$ must be connected
by a curve, and hence $V_{1,2}$ is infinite (and similarly
for~$V_{2,1}$).
 
%
%

It remains to define $V_{1,2}$ and $V_{2,1}$ and to verify the
claimed properties.  Let $F_1=F_1(\xx_1)$ be the formula
$\bigwedge_{\aa_2\in A_2} \Phi(\xx_1,\aa_2)$,  and let $V_1\subseteq\R^d$
be the variety $\{\xx_1\in\R^d: F_1(\xx_1)\}$.
We define  $F_2(\xx_2):=\bigwedge_{\aa_1\in A_1} \Phi(\aa_1,\xx_2)$
and $V_2$ similarly. We have $A_i\subseteq V_i$, $i=1,2$, by the assumption.

We set
$$
V_{1,2}:= \Bigl\{\xx_1\in V_1:(\forall\xx_2\in V_2)\, \Phi(\xx_1,\xx_2)\Bigr\}
$$
and
$$
V_{2,1}:= \Bigl\{\xx_2\in V_2:(\forall\xx_1\in V_1)\, \Phi(\xx_1,\xx_2)\Bigr\}.
$$
By definition, $\Phi$ holds everywhere on $V_{1,2}\times V_2$,
as well as on $V_1\times V_{2,1}$. Since $V_{1,2}\subseteq V_1$ and
$V_{2,1}\subseteq V_2$, we get that $\Phi$ holds 
everywhere on $V_{1,2}\times V_{2,1}$. 

It is easy to see that $A_1\subseteq V_{1,2}$
and $A_2\subseteq V_{2,1}$, and we now want to
argue that $V_{1,2}$ and $V_{2,1}$ can be defined
by bounded-size formulas.

First, $V_1$ is the variety defined
by the polynomials $f(\xx_1,\aa_2)$, where $f$
is one of the polynomials in $\Phi$ and $\aa_2\in A_2$.
By the bounded-dimension argument, only a bounded number
of these polynomials suffice to define $V_1$, and hence
we can obtain a bounded-size formula $\tilde F_1$ 
that is equivalent to~$F_1$.
Similarly, from $F_2$ we obtain a bounded-size
equivalent formula~$\tilde F_2$.

Now we consider the formula $F_{1,2}(\xx_1)$ defining $V_{1,2}$.
By definition,
$$
F_{1,2}(\xx_1):= F_1(\xx_1)\wedge \forall \xx_2 (F_2(\xx_2)
\Rightarrow \Phi(\xx_1,\xx_2)).
$$
We now replace $F_1$ and $F_2$ with $\tilde F_1$ and $\tilde F_2$,
which gives a bounded-size formula describing $V_{1,2}$, and then
we perform quantifier elimination for $\tilde F_2$ to obtain a
quantifier-free formula $\bar F_{1,2}(\xx_1)$ that is equivalent
to $F_{1,2}(\xx_1)$ and thus also defines~$V_{1,2}$. Moreover,
by the properties of available quantifier-elimination methods 
for the first-order theory of the reals
(see \cite[Sec.~11.3]{BasuPollackRoy-book}), quantifier elimination
for a bounded-size formula again yields a bounded-size formula, and so
the size of $\bar F_{1,2}$ is still bounded. Hence we can bound
the number of connected components of $V_{1,2}$
as in the plan above, and the considered special case of 
Theorem~\ref{t:multip} is proved.

\medskip

Next, we still assume that $\Phi=\Phi(\xx_1,\ldots,\xx_k)$ is algebraic,
but it can be $k$-ary for any~$k$. The basic idea is the same as before,
but this time we define a more complicated sequence of sets
$V_{i,J}$, with $i\in [k]$ and $J\subseteq [k]\setminus\{i\}$,
inductively by
\begin{eqnarray*}
V_{i,\emptyset}&:=&\Bigl\{\xx_i\in\R^d: (\forall \aa_1\in A_1)\cdots
(\forall \aa_{i-1}\in A_{i-1})(\forall \aa_{i+1}\in A_{i+1})\cdots
(\forall \aa_k\in A_k) \\
&&\ \ \ \ 
\Phi(\aa_1,\ldots,\aa_{i-1},\xx_i,\aa_{i+1},\ldots,\aa_k)\Bigr\}
\end{eqnarray*}
and, for $J\ne\emptyset$,
\begin{eqnarray*}
V_{i,J}&:=&\Big\{\xx_i\in \bigcap_{J'\subset J} V_{i,J'}:
(\forall\xx_j\in V_{j,J\setminus\{j\}},j\in J)
(\forall\aa_j\in A_j, j\in [k]\setminus \{i\}\setminus J)\\
&&\ \ 
\Phi(\xx_i;\xx_j,j\in J;\aa_j,j\in [k]\setminus \{i\}\setminus J)\Big\}.
\end{eqnarray*}
where the sequence of arguments of $\Phi$
means that the $i$th argument equals $\xx_i$,
the $j$th argument equals $\xx_j$ for all $j\in J$,
and the $j$th argument equals $\aa_j$ for 
all $j\in [k]\setminus \{i\}\setminus J$.

By definition, for $J'\subseteq J$ we have $V_{i,J}\subseteq V_{i,J'}$,
and this then shows that $\Phi$ holds everywhere on 
$V_{1,[k]\setminus\{1\}}\times\cdots\times V_{k,[k]\setminus\{k\}}$.
Inductively it is also easy to show that $A_i\subseteq V_{i,J}$ for all~$J$.

It remains to show that each $V_{i,J}$ can be described by
a bounded-size formula, possibly involving
quantifiers (then we apply  quantifier elimination and bound
the number of connected components as in the case $k=2$). 
We are going to show this by induction on $|J|$. 

Thus, the inductive hypothesis is that
we have bounded-size formulas $\tilde F_{i,J}(\xx_i)$,
possibly with quantifiers, describing $V_{i,J}$,
for all $J$ up to some size $s$, and we want to get such a formula
for $J$ of size~$s+1$.

By the definition, we have the following formula defining $V_{i,J}$:
\begin{eqnarray*}
F_{i,J}(\xx_i)&=&\biggl(\bigwedge_{J'\subset J}\tilde F_{i,J'}(\xx_i)\biggr)\\
&& 
\wedge (\forall \xx_j,j\in J)\biggl[
\biggl(\bigwedge_{j\in J} \tilde F_{j,J\setminus\{j\}}(\xx_j)\biggr)
 \Rightarrow G_{i,J}(\xx_i;\xx_j:j\in J)\biggr],
\end{eqnarray*}
where $G_{i,J}(\xx_i;\xx_j,j\in J)$ is the formula asserting
that $\Phi(\xx_i;\xx_j,j\in J;\aa_j,j\in [k]\setminus J\setminus\{i\})$
holds for all $\aa_j\in A_j$, $j\in [k]\setminus J\setminus\{i\}$. Now $G_{i,J}$ can be replaced with a bounded-size formula by the bounded-dimension
argument, and this yields the desired bounded-size
formula $\tilde F_{i,J}$, finishing the induction step.
We have proved the desired result for algebraic predicates.
\medskip

Finally, we let
$\Phi=\Phi(\xx_1,\ldots,\xx_n)$ be an arbitrary semialgebraic
predicate.
We may assume it to be of the form 
$\Phi=\Phi_1\vee\Phi_2\vee\cdots\vee \Phi_m$,
where each $\Phi_i$ is a conjunction of polynomial equations and
strict inequalities. Given large sets $A_1,\ldots,A_k$
such that $\Phi$ holds on $A_1\times\cdots\times A_k$,
by the $k$-partite Ramsey theorem we get that some $\Phi_\ell$ holds everywhere
on $A_1'\times\cdots\times A'_k$, where the $A_i'\subseteq A_i$ are
still large. (This is the only step where a dependence of $N$
on $\Phi$ enters, and so if $\Phi$ is a conjunction of atoms
and this step is skipped, $N$ depends only on $d$, $D$, and $k$ as claimed.)

Now we assume that $\Phi=\Phi_{=}\wedge\Phi_{<}$, where $\Phi_=$
is a conjunction of polynomial equations and $\Phi_{<}$ is a conjunction
of strict polynomial inequalities, and that $\Phi$ holds everywhere on
$A_1\times\cdots\times A_k$.

By the argument above for the $k$-partite algebraic case,
we obtain curves $Y_1,Y_2,\ldots,Y_k$,
where each $Y_i$ connects two distinct points of $A_i$,
and 
such that $\Phi_=$ holds on $Y_1\times\cdots\times Y_k$.
In particular, there are
$\aa_i\in A_i$, $i\in [k]$, such that each $\aa_i$ is a cluster
point of $Y_i$. Since $\Phi_<$ holds at $(\aa_1,\ldots,\aa_k)$,
it also holds for all $(\yy_1,\ldots,\yy_k)$ from a small
neighborhood of $(\aa_1,\ldots,\aa_k)$. Thus, we can choose
infinite sets $X_i\subseteq Y_i$ such that $\Phi$ holds everywhere
on $X_1\times\cdots\times X_k$ as desired. Theorem~\ref{t:multip}
is proved.
\end{proof}

As was mentioned in the introduction, the analog
of Theorem~\ref{t:multip} also holds for algebraic predicates
over the complex numbers. To see this, only two ingredients
in the proof above need to be changed: first, quantifier elimination
for the first-order theory of the reals needs to be replaced by
quantifier elimination for the first-order theory of algebraically
closed fields  (see \cite{puddu-sabia} for a recent work
and references), and second, we bound the number
of connected components of a variety in $\C^t$ by regarding
it as a semialgebraic set in $\R^{2t}$ and applying
the same bounds as in the proof above.

\heading{Algebraic predicates on sequences. }
Theorem~\ref{t:algpred} follows easily from the next two lemmas.
The first lemma is very similar to Theorem~\ref{t:multip},
but the sets $A_1,\ldots,A_k$ in (i), as well
as the sets $X_1,\ldots,X_k$ in (ii), are all equal.

\begin{lemma}\label{l:samefact} 
In the setting of Theorem~\ref{t:multip},
if a $d$-di\-mension\-al $k$-ary
semialgebraic predicate $\Phi$ that is a conjunction of 
polynomial equations and inequalities and
involves polynomials of degree at most $D$
is assumed to hold everywhere on $A^k=A\times\cdots\times A$,
where $|A|= N=N(d,k,D)$, then there is an infinite $X$
such that $\Phi$ holds everywhere on~$X^k$. For an arbitrary
$d$-di\-mension\-al $k$-ary semialgebraic $\Phi$, a similar statement
holds, only with $N$ depending on the length of $\Phi$ as well.
\end{lemma}

\begin{proof} First we assume $\Phi$ algebraic. 
We consider the ``symmetrization'' $\Phi_{\rm sym}$
of $\Phi$:
\[
\Phi_{\rm sym}(\xx_1,\ldots,\xx_k):=\bigwedge_{\pi\in S_k} 
\Phi(\xx_{\pi(1)},\ldots,\xx_{\pi(k)})
\]
(the conjunction is over all permutations of $[k]$). Then $\Phi_{\rm sym}$
also holds everywhere on $A^k$. When we define the sets $V_{i,J}$
for $\Phi_{\rm sym}$ as in the proof of Theorem~\ref{t:multip} above,
we have that $\Phi_{\rm sym}$, and hence $\Phi$, holds everywhere
on $\prod_{i=1}^k V_{i,[k]\setminus i}$. By the symmetry of
$\Phi_{\rm sym}$,  the $V_{i,J}$
satisfy $V_{\pi(i),\pi(J)}=V_{i,J}$ for every permutation
$\pi$, and in particular, the varieties $V_{i,[k]\setminus i}$
are all equal. Denoting their common value by $V$, we get that
$\Phi$ holds everywhere on $V^k$. The infinitude of 
$V$ follows by bounding the number of components as above.

If $\Phi$ is semialgebraic, then we again proceed as in the proof
of Theorem~\ref{t:multip}, observing that the choice of $X_1,\ldots,X_k$
can also be done symmetrically, i.e., with $X_1=\cdots=X_k$.
\end{proof}

The second lemma shows that for algebraic predicates, we can pass
from ``holding everywhere on a long sequence'' to ``holding everywhere
on a large Cartesian power''.

\begin{lemma} For every $d,k,D,n$ there exists $N$ 
(for which an explicit bound can be given) with the following
property. Assuming that a $d$-di\-mension\-al $k$-ary  algebraic predicate
$\Phi$
involving polynomials of degree at most $D$
holds everywhere on a sequence $\seq\aa$ of length $N$, with
all terms distinct,
then there is an $n$-element set $B\subseteq\seq\aa$
such that $\Phi$ holds everywhere on $B^k$
(here the inclusion $B\subseteq\seq\aa$ means that $B$ consists
of some of the terms of the sequence $\seq\aa$).
\end{lemma}

\begin{proof} The proof is again based on
the bounded-dimension argument.
Let us set $\seq\bb^{(0)}:= \seq\aa$, and let
us assume inductively that $\seq\bb^{(j-1)}$ is 
a subsequence of $\seq\aa$ of length $n_{j-1}$ such that
$\Phi(\bb_{i_1},\ldots,\bb_{i_k})$ holds for all choices
of $i_1,\ldots,i_{j-1}\in[n_{j-1}]$ and
$1\le i_{j}<i_{j+1}<\cdots<i_k\le n_{j-1}$.

For $i=1,2,\ldots,n_{j-1}$, let 
\begin{eqnarray*}
W^{(j)}_i&:=&\Bigl\{(\xx_1,\ldots,\xx_j)\in(\R^d)^j: 
(\forall i_{j+1},\ldots,i_k, i< i_{j+1}<\cdots<i_k)\\
&&\ \ \ \ \Phi(\xx_1,\ldots,\xx_j,
\bb_{i_{j+1}},\ldots,\bb_{i_k})\Bigr\}.
\end{eqnarray*}
The inductive assumption gives $(\bb_{i_1},\ldots,\bb_{i_j})\in
W^{(j)}_i$ for all $i_1,\ldots,i_{j-1}\in [n_{j-1}]$
and all $i_j\in [i]$.

We have $W^{(j)}_1\subseteq W^{(j)}_2\subseteq \cdots$, and
by the bounded-dimension argument, 
there are only a bounded number of distinct 
varieties among the $W^{(j)}_i$. Thus, there exist 
$i_{\rm min}<i_{\rm max}$, with $i_{\rm max}-i_{\rm min}$ large,
such that $W^{(j)}_{i_{\rm min}}=W^{(j)}_{i_{\rm max}}$.

For all $i_1,\ldots,i_{j-1}\in [n_{j-1}]$,
we have  $(\bb_{i_1},\ldots,\bb_{i_j})\in W^{(j)}_{i_{\rm max}}
=W^{(j)}_{i_{\rm min}}$
for all $i_j\in [i_{\rm min}+1,i_{\rm max}]$, and by the definition
of $W^{(j)}_{i_{\rm min}}$, we obtain that $\Phi(\bb_{i_1},\ldots,\bb_{i_k})$
holds for all $i_1,\ldots,i_j\in [i_{\rm min}+1,i_{\rm max}]$
and all $i_{j+1},\ldots,i_k$ with  $i_{\rm min}<i_{j+1}<\cdots<i_k$. Hence
we can let $\seq\bb^{(j)}$ be the subsequence of $\seq\bb^{(j-1)}$
indexed by $[i_{\rm min}+1,i_{\rm max}]$. This finishes the induction step.

The proof is concluded by letting $B$ be the set of elements of 
$\bb^{(k)}$. 
\end{proof}

\section*{Acknowledgments}
We would like to thank
Uri Andrews for help with logical vocabulary and for kindly scanning 
Rosenthal's thesis, and  Ehud Hrushovski
for a discussion. We are grateful to David Rosenthal for permission
to reproduce his thesis. Finally, we thank three anonymous referees
for careful reading and pointing out embarrassingly many small mistakes.

\bibliographystyle{plain}
\bibliography{cg,es}

\begin{thebibliography}{10}

\bibitem{AlonPPRS05}
N.~Alon, J.~Pach, R.~Pinchasi, R.~Radoi\v{c}i\'c, and M.~Sharir.
\newblock Crossing patterns of semi-algebraic sets.
\newblock {\em J. Comb. Theory, Ser. A}, 111(2):310--326, 2005.

\bibitem{bfl-nhp-90}
I.~B\'{a}r\'any, Z.~{F\"uredi}, and L.~Lov\'asz.
\newblock On the number of halving planes.
\newblock {\em Combinatorica}, 10:175--183, 1990.

\bibitem{BasuPollackRoy-book}
S.~Basu, R.~Pollack, and M.-F. Roy.
\newblock {\em Algorithms in real algebraic geometry}.
\newblock Algorithms and Computation in Mathematics 10. Springer, Berlin, 2003.

\bibitem{BoCoRo}
J.~Bochnak, M.~Coste, and M.-F. Roy.
\newblock {\em Real algebraic geometry}.
\newblock Springer-Verlag, Berlin, 1998.

\bibitem{bukh-al-sensitive-apx}
B.~Bukh, P.-S. Loh, and G.~Nivasch.
\newblock One-sided epsilon-approximants.
\newblock Manuscript in preparation, 2012.

\bibitem{BMN10}
B.~Bukh, J.~Matou\v{s}ek, and G.~Nivasch.
\newblock {Stabbing simplices by points and flats}.
\newblock {\em Discrete Comput. Geom.}, 43(2):321--338, 2010.

\bibitem{BMN11}
B.~Bukh, J.~Matou\v{s}ek, and G.~Nivasch.
\newblock {Lower bounds for weak epsilon-nets and stair-convexity}.
\newblock {\em Isr. J. Math.}, 182:199--228, 2011.

\bibitem{cfpss-semialg}
D.~Conlon, J.~Fox, J.~Pach, B.~Sudakov, and A.~Suk.
\newblock Ramsey-type results for semi-algebraic relations.
\newblock {\em Trans. Amer. Math. Soc.}, 2013.
\newblock To appear. Preprint arXiv:1301.0074. Extended abstract in \emph{Proc.
  29th Annual ACM Symposium on Computational Geometry}, Rio de Janeiro, Brazil,
  2013.

\bibitem{eli-mat}
M.~Eli\'a\v{s} and J.~Matou\v{s}ek.
\newblock Higher-order {Erd\H{o}s--Szekeres} theorems.
\newblock In {\em Proc. ACM Sympos. Comput. Geom}, 2012.
\newblock Also in arXiv:1111.3824.

\bibitem{es-cpg-35}
P.~Erd{\H o}s and G.~Szekeres.
\newblock A combinatorial problem in geometry.
\newblock {\em Compositio Math.}, 2:463--470, 1935.

\bibitem{fox-al-geomexp}
J.~Fox, M.~Gromov, V.~Lafforgue, A.~Naor, and J.~Pach.
\newblock Overlap properties of geometric expanders.
\newblock {\em J. reine angew. Math. (Crelle)}, 2012.
\newblock In press, available on-line.

\bibitem{GabriVor}
A.~Gabrielov and N.~Vorobjov.
\newblock Betti numbers of semialgebraic sets defined by quantifier-free
  formulae.
\newblock {\em Discrete Comput. Geom.}, 33(3):395--401, 2005.

\bibitem{gerken-6h}
T.~Gerken.
\newblock On empty convex hexagons in planar point sets.
\newblock In {\em J.\,E.~Goodman, J.~Pach, R.~Pollack (eds.), Twentieth
  Anniversary Volume: Discrete \& Computational Geometry}, pages 1--34.
  Springer, New York, NY, 2008.

\bibitem{grs-rt-90}
R{.\,L.} Graham, B{.\,L.} Rothschild, and J.~Spencer.
\newblock {\em Ramsey Theory}.
\newblock J. Wiley \& Sons, New York, 1990.

\bibitem{Mat-dg}
J.~Matou\v{s}ek.
\newblock {\em Lectures on Discrete Geometry}.
\newblock Springer, New York, 2002.

\bibitem{MorrisSoltan}
W.~Morris and V.~Soltan.
\newblock {The Erd\H{o}s--Szekeres problem on points in convex position---a
  survey}.
\newblock {\em Bull. Amer. Math. Soc., New Ser.}, 37(4):437--458, 2000.

\bibitem{nicolas}
C.~M. Nicol\'{a}s.
\newblock The empty hexagon theorem.
\newblock {\em Discr. Comput. Geom.}, 38(2):389--397, 2007.

\bibitem{puddu-sabia}
S.~Puddu and J.~Sabia.
\newblock {An effective algorithm for quantifier elimination over algebraically
  closed fields using straight line programs}.
\newblock {\em J. Pure Appl. Algebra}, 129(2):173--200, 1998.

\bibitem{Rosenthal-thesis}
D{.\,A.} Rosenthal.
\newblock The classification of the order indiscernibles of real closed fields
  and other theories.
\newblock PhD.~thesis, University of Wisconsin--Madison, 1981.
\newblock Available at \url{http://www.borisbukh.org/rosenthal_thesis.pdf}.

\bibitem{steele-surv}
M.~J. Steele.
\newblock Variations on the monotone subsequence theme of {Erd\H{o}s and
  Szekeres}.
\newblock In {\em D.~Aldous et al., editors, Discrete Probability and
  Algorithms, IMA Volumes in Mathematics and its Applications 72}, pages
  111--131. Springer, Berlin etc., 1995.

\bibitem{t-dmeag-51}
A.~Tarski.
\newblock {\em A decision method for elementary algebra and geometry}.
\newblock Univ. of California Press, Berkeley, CA, 1951.

\bibitem{vz-ctpci-92}
S.~Vre{\'c}ica and R.~{\v Z}ivaljevi{\'c}.
\newblock The colored {T}verberg's problem and complexes of injective
  functions.
\newblock {\em J. Combin. Theory Ser. A}, 61:309--318, 1992.

\end{thebibliography}

\end{document}